\documentclass[a4paper,12pt]{amsart}
\usepackage[english]{babel}
\usepackage{amsthm,amsmath,amssymb}
\def\O{\mathcal{O}}
\def\ms{{n_0}}
\newtheorem{rema}{Remark}

\newtheorem{lemma}{Lemma}
\newtheorem{conj}{Conjecture}
\newtheorem{corollary}{Corollary}

\newtheorem{thm}{Theorem}

\newcommand{\rmd}{\,\mathrm{d}}
\def \a   {\alpha}

\def \g   {\gamma}

\def \eps {\varepsilon}
\def\E{{\mathbb{E}}}
\def\N{{\mathbb{N}}}

\def\Var{{\mathtt{Var}}}
\def\P{{\mathbb{P}}}

\def\|{\,|\,}
\begin{document}
\def\bn#1\en{\begin{align*}#1\end{align*}}
\def\bnn#1\enn{\begin{align}#1\end{align}}
\newcounter{tally}

\title{Turning a coin over instead of tossing it}
\author[J. Engl\"ander]{J\'anos Engl\"ander}
\address[J. Engl\"ander]{Department of Mathematics\\ University of Colorado\\
Boulder, CO-80309-0395}
\urladdr{http://euclid.colorado.edu/\textasciitilde{}englandj/MyBoulderPage.html}
\thanks{The hospitality of Microsoft Research and the University of Washington is gratefully acknowledged by the first author.}

\author[S. Volkov]{Stanislav Volkov}
\address[S. Volkov]{Centre for Mathematical Sciences\\ Lund University\\ Lund 22100-118, Sweden}
\urladdr{http://www.maths.lth.se/~s.volkov/}
\thanks{Research of the second author was supported in part by the  Swedish Research Council grant VR2014--5157.
}

\begin{abstract}
Given a sequence of numbers $\{p_n\}$ in $[0,1]$, consider the following experiment. First, we flip a fair coin and then, at step $n$, we turn the coin over to the other side with probability $p_n$, $n\ge 2$.  What can we say about the distribution of the empirical frequency of heads as $n\to\infty$?

We show that a number of phase transitions take place as the turning gets slower (i.~e.~$p_n$ is getting smaller), leading first to the breakdown of the Central Limit Theorem and then to that of the Law of Large Numbers. It turns out that the critical regime is $p_n=\text{const}/n$. Among the scaling limits, we obtain Uniform, Gaussian, Semicircle and Arcsine laws.
\end{abstract}
\maketitle
\section{General model} 
In this paper we examine what happens if, instead of tossing a coin, we turn it over (from heads to tails and from tails to heads), with certain probabilities.

To define the model precisely, let $p_n$, $n=1,2,\dots$  be a given deterministic sequence of numbers between $0$ and $1$. We define the following time-dependent `coin turning process' $X_n\in\{0,1\}$, $n\ge 1$, as follows. Let $X_1=1$ (`heads') or $=0$ (`tails') with probability $1/2$. For $n\ge 2$,  set recursively
$$
X_n:=\begin{cases}
1-X_{n-1},&\text{with probability } p_n;\\
X_{n-1},&\text{otherwise},
\end{cases}
$$
that is, we turn the coin over with probability $p_n$ and do nothing with probability $1-p_n$.

Consider $ \frac1N \sum_{n=1}^N X_n$, that is, the  empirical frequency of $1$'s (`heads') in the sequence of $X_n$'s. We are interested in the asymptotic behavior, in law, of this random variable as $N\to\infty$.

Since we are interested in limit theorems, we center the variable $X_n$; for convenience, we also multiply it by two, thus focus on $Y_n:=2X_n-1\in\{-1,+1\}$ instead of $X_n$. We have  
$$
Y_n:=\begin{cases}
-Y_{n-1},&\text{with probability } p_n;\\
Y_{n-1},&\text{otherwise}.
\end{cases}
$$
Note that the sequence $\{Y_n\}$ can be defined equivalently as follows.

Let $$Y_n:= (-1)^{\sum_1^n W_i},$$ where $W_1,W_2,W_3,...$ are independent Bernoulli variables with parameters $p_1,p_2,p_3,...$, respectively, and $p_1=1/2$. 
\begin{rema}[Poisson binomial random variable]
The number of turns that occurred up to $n$, that is $\sum_2^n W_i$, 
is a {\it Poisson binomial random variable}. The Poisson binomial distribution has many applications in different areas such as reliability, actuarial science, survey sampling, econometrics, and so on. Its characteristic function
is fairly simple: 
$$
\varphi(t)=\prod_{j=2}^n [1-p_j+p_j \exp(it)].
$$ 
See~\cite{Hong11} for more on Poisson binomial distribution;  see also~\cite{Volkova}. 
\end{rema}

The following quantity will play an important role: for $1\le i<j\le N$, let 
$$
e_{i,j}:=\prod_{k=i+1}^j (1-2p_k).
$$

For the centred variables $Y_n$, we have $Y_j=Y_i(-1)^{\sum_{i+1}^jW_k},\ j>i$, and so, using $\mathtt{Corr}$ and $\mathtt{Cov}$ for correlation and covariance, respectively, one has
\begin{align}\label{eq_eij}
&\mathtt{Corr}(Y_i,Y_j)=\mathtt{Cov}(Y_i,Y_j)=\E( Y_iY_j)=\E(-1)^{\sum_{i+1}^jW_k}=\prod_{i+1}^j \E (-1)^{W_{k}}=\prod_{k=i+1}^j (1-2p_k)=e_{i,j};\\
&\E (Y_j\mid Y_i)=Y_i \E(-1)^{\sum_{i+1}^jW_k}=e_{i,j}Y_i.
\end{align}
\begin{corollary}[Correlation estimate]\label{Corr.estim} Assume that $p_k\to 0$ and let $n_0$ be such that $p_k\le 1/2$ for $k\ge n_0$.
For $n_0\le i<j$,
$$
\exp\left(-2\sum_{i+1}^j p_k\right)\cdot \prod_{i+1}^j (1-r_k)\le e_{i,j}\le \exp\left(-2\sum_{i+1}^j p_k\right),
$$
where $r_k:=2p_k^2 e^{2p_k}$, which is tending to zero rapidly.

Furthermore, for any given $C>1$ there exists an $\ms$ such that
For $\ms\le i<j$,
$$
\exp\left(-2\sum_{i+1}^j C p_k\right)\le e_{i,j}
\le \exp\left(-2\sum_{i+1}^j p_k\right),
$$

\end{corollary}
\begin{proof}
Use the Remainder Theorem for Taylor series, yielding 
$$
0\le e^{-2p_k}-(1-2p_k)\le 2p_k^2,
$$
that is,
$$
\exp\left(-2 p_k\right)\cdot (1-r_k)\le 1-2p_k\le \exp\left(-2 p_k\right),
$$
and multiply these inequalities, to get the first statement.

For the second statement, use that for sufficiently small positive $x$,
$$
e^{-Cx}\le 1-x\le e^{-x}.
$$
\end{proof}
Similarly to~\eqref{eq_eij}, if $K=2m$ is a positive even number, and $i_1<i_2<\dots < i_K$ then, using the fact that
$$
\sum_{k=1}^{i_{1}}W_{k}+\sum_{k=1}^{i_{2}}W_{k}+...+\sum_{k=1}^{i_{K}}W_{k}
=\sum_{k=i_1+1}^{i_{2}}W_{k}+\sum_{k=i_{3}+1}^{i_{4}}W_{k}+...+\sum_{k=i_{K-1}+1}^{i_{K}}W_{k}
\mod 2,
$$
we obtain that
\begin{align}\label{eqEEE}
\E (Y_{i_1}Y_{i_2}\dots Y_{i_K})&=
\E(-1)^
{\sum_{1}^{i_{1}}W_{k}+\sum_{1}^{i_{2}}W_{k}+...+\sum_{1}^{i_{K}}W_{k}}
=\E(-1)^{\sum_{i_1+1}^{i_{2}}W_{k}+\sum_{i_{3}+1}^{i_{4}}W_{k}+...+\sum_{i_{K-1}+1}^{i_{K}}W_{k}}
\nonumber\\
&=
\E (-1)^{\sum_{i_{1}+1}^{i_{2}}W_{k}}\cdot
\E (-1)^{\sum_{i_{3}+1}^{i_{4}}W_{k}}
\cdot ...\cdot
\E (-1)^{\sum_{i_{K-1}}^{i_{K}}W_{k}}=e_{i_{1},i_2} e_{i_3,i_4} \dots e_{i_{K-1},i_K}.
\end{align}
We also define $S_N=Y_1+\dots+Y_N$, and note that from symmetry it follows that if~$K$ is a positive odd integer, then~$\E S_N^K=0$.

We close this section with introducing some frequently used notation.

{\bf Notations:}
In the sequel, ${\sf Bessel\,I}_\a$ and ${\sf Bessel\,K}_\a$ will denote the modified Bessel function of the first kind (or Bessel-I function) and the modified Bessel function of the second kind (or Bessel-K function), respectively.

Writing out these functions explicitly, one has
$$
{\sf Bessel\,I}_\a(x)=\sum_{m=0}^{\infty}
\frac{1}{m!\Gamma(m+\alpha+1)}\left(\frac x2\right)^{2m+\alpha},
$$
and
$$
{\sf Bessel\,K}_\a(x)=\frac{\pi}{2}\frac{{\sf Bessel\,I}_{-\a}(x)-{\sf Bessel\,I}_\a(x)}{\sin(\alpha \pi)},
$$
if $\a$ is not an integer (otherwise it is defined through the limit),
where $\Gamma$ is Euler's gamma-function.
See e.g., Sections 9--10 in~\cite{AS},  and formula~(6.8) in~\cite{IBF}.

\section{Supercritical cases}\label{secSuper}
First, if $\sum_n p_n<\infty$ then by the Borel-Cantelli lemma, only finitely many turns will occur a.s.; therefore the side we see stabilizes and by the assumption on~$X_1$,

$$
\frac{X_1+\dots+X_N}{N} \to \zeta\text{ a.s.}
$$
where $\zeta$ is a ${\sf Bernoulli}\left (\frac{1}{2}\right)$ random variable.

\section{A simple critical case}
Assume that
$$
p_n=\frac{1}{n},\quad  n\ge 2.
$$
\begin{thm}
The law of $S_N/N$ converges  to ${\sf Uniform}([-1,1])$ as $N\to\infty$.
\end{thm}
\begin{rema} The reader can easily check that convergence in distribution cannot be strengthened to an almost sure one.
\end{rema}
\begin{proof}
Equation~\eqref{eq_eij} gives
\begin{align*}
e_{i,j}=
\left(1-\frac 2j\right)\left(1-\frac 2{j-1}\right)
\dots \left(1-\frac 2{i+1}\right)=\frac{(i-1)i}{(j-1)j}.
\end{align*}
 Therefore from~\eqref{eqEEE} we obtain that for even positive $K$,
\begin{align*}
\E (Y_{i_1}Y_{i_2}\dots Y_{i_K})=
\frac{i_{1} (i_1-1)}{i_{2}(i_2-1)}\cdot
\frac{i_3(i_3-1)}{i_4(i_4-1)}\cdot \dots\cdot
\frac{i_{K-1}(i_{K-1}-1)}{i_{K}(I_K-1)}.
\end{align*}
Now recall that $S_N=Y_1+\dots +Y_{N}$  and  that the distribution of  $Y_n$, and thus of~$S_N$ are symmetric around~$0$. Hence, the  odd moments of $S_N$ are zero: $\E S_N^K=0,\ K=1,3,5,\dots$

For $K$ even, we can use the multinomial theorem:
$$
S_N^K=I+K!\sum_{1\le i_1<i_2<\dots < i_K\le N}
\E (Y_{i_1}Y_{i_2}\dots Y_{i_K}),
$$
where $I$ stands for the sum of products where not all terms are different. Note  that  $|Y_i^l|\le 1$ for any $l\ge 1$
(and $Y_i^l\equiv 1$ for $l$ even). Therefore $|I|\le m(N,K)$, where $m(N,K)$ is the number of such products. But $m(N,K)\le N\cdot N^{K-2}=N^{K-1}$, because each such product can be written (not uniquely) as $Y_{i_{\ell}}^2 \cdot Y_{i_{1}}Y_{i_{2}}\dots Y_{i_{K-2}}=Y_{i_{1}}Y_{i_{2}}\dots Y_{i_{K-2}}$, where  the numbers $i_{\ell},i_1,...,i_{K-2}$ are between $1$ and $N$ and are not necessarily distinct. Hence,
\begin{equation}\label{even.moments}
\E S_N^K=
K!\sum_{1\le i_1<i_2<\dots < i_K\le N}
\E (Y_{i_1}Y_{i_2}\dots Y_{i_K}) + \O(N^{K-1}).
\end{equation}
We can estimate the  sum  in~\eqref{even.moments} as follows:
\begin{align*}
\sum_{1\le i_1<i_2<\dots < i_K\le N}&\E (Y_{i_1}Y_{i_2}\dots Y_{i_K})
=
\sum_{
{1\le i_1<i_2<\dots < i_K\le N}
}
\frac{i_{1} (i_1-1)}{i_{2}(i_2-1)}\cdot
\frac{i_3(i_3-1)}{i_4(i_4-1)}\cdot \dots\cdot
\frac{i_{K-1}(i_{K-1}-1)}{i_{K}(I_K-1)}
=:(*)
\end{align*}
Summing first over $i_1$, then over $i_2$, etc., gives
\begin{align*}
(*)&= \sum_{2\le i_2<i_3<\dots < i_K\le N} \frac{i_2(i_2-1)(i_2-2)}{3}
\times \frac{1}{i_{2}(i_2-1)}\cdot
\frac{i_3(i_3-1)}{i_{4}(i_4-1)}\cdot \dots\cdot
\frac{i_{K-1}(i_{K-1}-1)}{i_{K}(i_K-1)}
\\
&=
\frac 2{3!} \sum_{3\le i_2<\dots < i_K\le N}
(i_2-2)\cdot
\frac{i_3(i_3-1)}{i_{4}(i_4-1)}\cdot \dots\cdot
\frac{i_{K-1}(i_{K-1}-1)}{i_{K}(i_K-1)}
\\
&=
\frac 2{3!} \sum_{3\le i_3<i_4<\dots < i_K\le N} \frac{(i_3-2)(i_3-3)}{2}
\times
\frac{i_3(i_3-1)}{i_{4}(i_4-1)}\cdot \dots\cdot
\frac{i_{K-1}(i_{K-1}-1)}{i_{K}(i_K-1)}
\end{align*}
\begin{align*}
&=
\frac 1{3!} \sum_{3\le i_3 <i_4<\dots < i_K\le N}
[(i_3-3)(i_3-2)(i_3-1)i_3]
\times
\frac{1}{i_{4}(i_4-1)}\cdot \dots\cdot
\frac{i_{K-1}(i_{K-1}-1)}{i_{K}(i_K-1)}
\\
&=
\frac 4{5!}
\sum_{4\le i_4<i_5<\dots < i_K\le N} (i_4-2)(i_4-3)(i_4-4)
\times
\frac{i_5(i_5-1)}{i_{6}(i_6-1)}\cdot \dots\cdot
\frac{i_{K-1}(i_{K-1}-1)}{i_{K}(i_K-1)}
\\
&=
\frac 1{5!}
\sum_{5\le i_5<\dots < i_K\le N} [i_5(i_5-1)\dots(i_5-5)]
\times
\frac{1}{i_{6}(i_6-1)}\cdot \dots\cdot
\frac{i_{K-1}(i_{K-1}-1)}{i_{K}(i_K-1)}
\\
&= \dots =
\frac K{(K+1)!}
\sum_{K\le  i_K\le N} i_K(i_K-1)\dots(i_K-(K-1))
\\
&= \frac{(N+1)N(N-1)(N-2)\dots (N-K+2))}{(K+1)!}=\frac{1}{N-K+1}\binom{N}{K+1}.
\end{align*}
It follows that, as $N\to\infty$,
$$
\sum_{1\le i_1<i_2<\dots < i_K\le N}\E (Y_{i_1}Y_{i_2}\dots Y_{i_K})=\frac{N^K}{(K+1)!}+\O(N^{K-1}),
$$
hence from ~\eqref{even.moments} we obtain
$$
\E S_N^K=
K!\frac{N^K}{(K+1)!}+\O(N^{K-1})=\frac{N^K}{K+1}+\O(N^{K-1}).
$$
Thus,  as $N\to\infty$,
$$
\E \left[\frac{S_N}{N}\right]^K=\frac{1}{K+1}+\O(N^{-1}).
$$
Putting things together, we have obtained that
$$ \lim_{N\to\infty}\E S_N^K
=\begin{cases}
0,&\text{ if $K$ is odd};\\
\frac{1}{K+1},&\text{ if $K$ is even.}
\end{cases}
$$
Hence the moments of $S_N/N$ converge to those of {\sf Uniform}$([-1,1])$. Since {\sf Uniform}$([-1,1])$ is supported on a compact interval, it follows (see e.g.\ Section~2, Exercise 3.27 in~\cite{DUR}) that it is the limit of the laws of $S_N/N$.
\end{proof}

\section{General critical case}
Fix $a>0$ and let
$$
p_n=\frac{a}n,\quad n\ge 2.
$$
Denote by ${\sf Beta}(a, a)$ the symmetric Beta distribution with $a>0$, with the following moment generating function
\begin{equation}\label{Bessel}
M_{{\sf Beta}(a, a)}(t)=1+\sum_{k=1}^{\infty}\left(\prod_{l=1}^{\infty}\frac{a+l-1}{2a+l-1}\right)\frac{t^k}{k!}=e^{t/2}\left(\frac{t}{4}\right)^{\frac 12-a}\Gamma\left(a+\frac12\right){\sf Bessel\, I}_{a-\frac12}\left(\frac t2\right).
\end{equation}

\begin{thm}
The law of $\frac 1N \sum_i^N {X_i}$ converges to ${\sf Beta}(a,a)$ as $N\to\infty$.
\end{thm}

\begin{proof}
From~\eqref{eq_eij} we get
\begin{align}\label{eij_gen_lin}
e_{i,j}&=\exp\left\{\sum_{n=i+1}^j \log \left(1-\frac {2a}n\right)\right\}
=\exp\left\{\O\left(\frac{j-i}{i^2}\right)-2a\sum_{n=i+1}^j \frac 1n\right\}
\\ &
=
\exp\left\{\O\left(\frac{j-i}{i^2}\right)-2a\log\left(\frac ji\right)\right\}
=  \frac{i^{2a}}{j^{2a}} \cdot \left(1+ \O\left(\frac{j-i}{i^2}\right)\right) ,\nonumber
\end{align}
and for even $K$ we  obtain
\begin{align*}
\E S_N^K&=
K!\sum_{1\le i_1<i_2<\dots < i_K\le N}
\E (Y_{i_1}Y_{i_2}\dots Y_{i_K})+\O(N^{K-1})
\\
&=K!
\sum_{1\le i_1<i_2<\dots < i_K\le N}
\frac{i_{1}^{2a}}{i_{2}^{2a}}\cdot
\frac{i_{3}^{2a}}{i_{4}^{2a}}\cdot \dots\cdot
\frac{i_{K-1}^{2a}}{i_{K}^{2a}}
+\O(N^{K-1})
\\
&=
\frac{K! \, N^K\left(1+\O(N^{-1})\right)}{(1+2a)\cdot 2\cdot(3+2a)\cdot 4\cdot\dots \cdot (K-1+2a)\cdot K}
+\O(N^{K-1}).
\end{align*}
Just like before, since we are working on a compact interval, we conclude that $S_N/N\to \xi_a$ in distribution,  where~$\xi_a$ is distributed on~$[-1,1]$ and has the following moments:
\begin{align*}
\E \left[\xi_{a}^K\right] =
\begin{cases}
 0,&K\text{ is odd;}\\
\displaystyle\frac{(2m)!}{2^m\cdot  m!\cdot (2a+1)(2a+3)\dots(2a+(2m-1))}
,&K=2m\text{ is even,}
\end{cases}
\end{align*}
which, for even moments, can be equivalently written as
\begin{align}\label{eq_mom}
\E \left[\xi_{a}^{2m}\right] =
\frac{(2m)!\, \Gamma(a+1/2)}{2^{2m} \Gamma(m+a+1/2)}.
\end{align}
The moment generating function of~$\xi_a$ is
$$
 M_a(t)=\E e^{t\xi_{a}}=1+\sum_{m=1}^{\infty} \frac{t^{2m}}{2^m\cdot m!\cdot
 \prod_{i=1}^m (2a+2i-1)}
 ={\sf Bessel\,I}_{a-1/2}(t)\Gamma(a+1/2) (t/2)^{1/2-a}.
$$

Let $\zeta_a:=(\xi_a+1)/2$. We know that $\frac 1N \sum_i^N {X_i}\to \zeta_a$ in distribution, and using~\eqref{Bessel}, 
$$
 \E e^{t\zeta_{a}} =e^{t/2}M_a(t/2)=e^{t/2}{\sf Bessel\,I}_{a-1/2}(t/2)\Gamma(a+1/2) (t/4)^{1/2-a}=M_{{\sf Beta}(a, a)}(t),
$$ 
completing the proof.
\end{proof}
\begin{rema}[Particular cases and densities]
Note that in particular, for $a=1, a=1/2$ and $a=3/2$, the limiting law is {\sf Uniform}$([0,1])$,  the Arcsine law, and the Wigner semicircle law on~$(-1,1)$, respectively.

Concerning the corresponding densities, we have the following explicit formulas.
\begin{itemize}
\item {\bf Arcsine Law:} Let~$a=1/2$. Then
$$
 \E e^{t\xi}=\sum_{m=0}^{\infty} \frac{t^{2m}}{(2^m\cdot m!)^2}
 ={\sf Bessel\,I}_0(t)=\frac{1}{\pi}\int_0^{\pi}e^{t\cos(\theta)} \rmd\theta;
$$
consequently (see e.g.~\cite{AS}, formula~29.3.60) $\xi_{1/2}$ has a density
$$
 f_{\xi_{1/2}}(x)=
\begin{cases}
\displaystyle \frac{1}{\pi\sqrt{1-x^2}}, & -1<x<1;\\
0, &\text{otherwise}.
\end{cases}
 $$

\item {\bf Semicircle law on $(-1,1)$:} Let  $a=3/2$. Then
$$
 \E e^{t\xi}=\frac{2\,{\sf Bessel\,I}_1(t)}{t}
 =\frac 1{\pi}\int_0^{\pi}e^{t\cos(\theta)} \cos(\theta)\rmd\theta
 ={\sf Bessel\, I}_0(t)-{\sf Bessel\, I}_2(t);
$$
consequently (see~\cite{AS}, formula~9.6.19) $\xi_{3/2}$ has a density
$$
 f_{\xi_{3/2}}(x)=
 \begin{cases}
 \displaystyle \frac{2}{\pi}\sqrt{1-x^2}, & -1<x<1;\\
 0, &\text{otherwise}.
 \end{cases}
$$

\item {\bf General case:}  The density of $\xi_a$ is given by
$$
 f_{\xi_a}(x)=\frac{\Gamma(a+1/2)}{\Gamma(a)\,\sqrt{\pi}}\left(1- x^2\right)^{a-1}
$$
for $-1<x<1$.
Indeed, for $m\in \N$ we have
\begin{align*}
&\int_{-1}^1 x^{2m} \left(1- x^2\right)^{a-1}\rmd x\\
&=\int_{0}^1 y^{m-1/2} \left(1- y\right)^{a-1}\rmd y
={\sf Beta}(m+1/2,a)= \frac{\Gamma(m+1/2)\Gamma(a)}{\Gamma(m+a+1/2)},
\end{align*}
which is consistent with the moments $\E \xi^{2m}$ given by~\eqref{eq_mom}.
\end{itemize}
\end{rema}

\section{Sub-critical case}
Now fix $\g,a>0$ and let
$$
p_n=\frac{a}{n^\g},\quad n\ge 2.
$$
Note that  $\g>1$ corresponds to the supercritical case studied in Section~\ref{secSuper}; so from now on assume $0<\g<1$.
\begin{thm}\label{gamma.thm}
The law of $S_N/N^{(1+\g)/2}$ converges to  ${\sf Normal}(0,\sigma^2)$
where
\begin{align}\label{eqsigma}
\sigma:=\frac{1}{\sqrt{a(1-\g)}}.
\end{align}
\end{thm}
\begin{proof} 
Let $\eta_{a,\g,N}:=S_N/N^{(1+\g)/2}$. Let $\eta_{a,\g}$ be  normally distributed with variance $\sigma^2$ where $\sigma$ is as in~\eqref{eqsigma}.
We will prove that $\lim_{N\to\infty}\eta_{a,\g,N}= \eta_{a,\g}$ in law.

In the proof we will use that 
\begin{equation}\label{moments.normal}
\E \eta_{a,\g}^K=\begin{cases}
0, &\text{ if $K$ is odd};\\
\sigma^K (K-1)!! &\text{ if $K$ is even}
\end{cases}
\end{equation}
(see e.g.~\cite{DUR}, Section 2, Exercise~3.15).

Let $A$ ($A>a$) be a given constant.
By Corollary \ref{Corr.estim}, there exists an $\ms=\ms(a,A,\g)$ such that for $\ms\le i<j$, one has
$$
\exp\left\{-2A\sum_{n=i+1}^j \frac{1}{n^\g}\right\}\le e_{i,j}\le \exp\left\{-2a\sum_{n=i+1}^j \frac{1}{n^\g}\right\}.
$$ 
Bounding the sum by the integral using the fact that~$x^{-\g}$ is decreasing, we have
\begin{align*}
\frac{(j+1)^{1-\g}-(i+1)^{1-\g}}{1-\g}=\int_{i+1}^{j+1}\frac{\rmd x}{x^\g} \le \sum_{n=i+1}^j \frac{1}{n^\g} &\le \int_i^j\frac{\rmd x}{x^\g}=
\frac{j^{1-\g}-i^{1-\g}}{1-\g}
\end{align*}
yielding
\begin{equation}\label{sum.int}
\exp\left(
-\frac{2A}{1-\g}\left[j^{1-\g}-i^{1-\g}\right]
\right)\le e_{i,j}\le \exp\left(
-\frac{2a}{1-\g}\left[(j+1)^{1-\g}-(i+1)^{1-\g}\right]
\right),
\end{equation}
that is, using the shorthands  $c:=2a(1-\g)^{-1}$ and $d:=2A(1-\g)^{-1}$,
$$
\exp\left(-d[j^{1-\g}-i^{1-\g}]\right)\le e_{i,j}\le \exp\left(-c[(j+1)^{1-\g}-(i+1)^{1-\g}]\right).
$$

It is easy to check  that $\sup_N \E \eta^2_{a,\g,N}<\infty$ (see the computation below with $m=1$), and thus Chebyshev's inequality implies that  $\{\eta_{a,\g,N}\}$ is a tight sequence of random variables. Hence, it is enough to show that each sub-sequential limit is the same.

Assume that $(N_l)_{l\ge 1}$ is a sub-sequence and $\lim_{l\to\infty}${\sf Law}($\eta_{a,\g,N_l})=\mathcal{L}$. 
Because trivially
$$
\left|\frac{Y_1+...+Y_{\ms-1}}{N_l^{\frac{1+\g}{2}}}\right|\le \frac{\ms-1}{N_l^{\frac{1+\g}{2}}},$$ one has $\mathcal{L}=\lim_{l\to\infty}\mathcal{L}_{N_{l},A}$ too, where $$\mathcal{L}_{N_{l},A}:={\sf Law}\left(\frac{Y_{\ms}+...+Y_{N_{l}}}{N_l^{\frac{1+\g}{2}}}\right),
$$
and in fact, this limit must be the same for any $A>a$ (and corresponding $\ms=\ms(a,A,\g)$). Informally, this just means that we can through away a finite chunk of the sequence of $Y_i$ (at the beginning) without affecting its limit.

Let us denote   the even moments of $\mathcal{L}$ by $M_{2m}\in [0,\infty]$, $m\ge 1$, while we note again that the odd moments must be zero by symmetry. Also, $M_{N_{l},A,K}$ will denote the $K$th moment under $\mathcal{L}_{N_{l},A}$.

We will show below that for a fixed~$A>a$ and $K=2m$, $m\ge 1$,
\begin{align}\label{two.sided.est}
\frac{(2m-1)!!}{[A(1-\g)]^{m}}&\le \liminf_{l\to\infty} M_{N_{l},A,K}=\liminf_{l\to\infty}\E\left[\frac{Y_{\ms}+...+Y_{N_{l}}}{N_l^{\frac{1+\g}{2}}}\right]^K\\  &\le \limsup_{l\to\infty} M_{N_{l},A,K}=\limsup_{l\to\infty}\E\left[\frac{Y_{\ms}+...+Y_{N_{l}}}{N_l^{\frac{1+\g}{2}}}\right]^K\le \frac{(2m-1)!!}{[a(1-\g)]^{m}}.\nonumber
\end{align}
Once~\eqref{two.sided.est} is shown, it will follow from the upper estimate and from the relation $\mathcal{L}=\lim_{l\to\infty}\mathcal{L}_{N_{l},A}$ for all $A>a$ that 
\begin{equation}\label{good}
\lim_{l\to\infty} M_{N_{l},A,K}=M_K
\end{equation}
for all $K\ge 1$ and all $A>a$. 
Since~\eqref{two.sided.est} holds for any $A>a$,  letting~$A\downarrow a$ and using~\eqref{two.sided.est} and~\eqref{good} that in fact
\begin{align*}
M_K=\frac{(2m-1)!!}{[a(1-\g)]^{m}}.
\end{align*}
In summary, we obtain that for any fixed $A>a$, 
\begin{equation}\label{hurray}
 \lim_{l\to\infty} M_{N_{l},A,K}=\frac{(2m-1)!!}{[a(1-\g)]^{m}}.
\end{equation}
At the same time the normal distribution is uniquely determined by its moments, and therefore  the convergence towards a normal law  is implied by  the convergence of all the moments (see e.g.~\cite{DUR}, Section 2.3.e). In our case, \eqref{hurray} along with \eqref{moments.normal} imply $\mathcal{L}=\lim_{l\to\infty}\mathcal{L}_{N_{l},A}={\sf Normal}(0,\sigma^2)$. 
Therefore, it only remains to prove~\eqref{two.sided.est}.

Let us start with the upper estimate in~\eqref{two.sided.est}.
It is easy to see that for~$K=2m$, one has
$$
\E\left[Y_{\ms}+...+Y_{N}\right]^K=I+
K!\sum_{\ms\le i_1<i_2<\dots < i_K\le N}
\E (Y_{i_1}Y_{i_2}\dots Y_{i_K})
$$
where $I$ are lower order terms, as it will be shown below.
Using~\eqref{eqEEE} along with~\eqref{sum.int}, we may continue with
\begin{align}\label{upper.est.moment}
&\le 
I+K!\sum_{\ms+1\le i_1<i_2<\dots < i_K\le N+1}
\exp\left(cU_{i_{1},...,i_{k}}\right),
\end{align}
where
$$
U_{i_{1},...,i_{k}}:=i_1^{1-\g}-i_2^{1-\g}+i_3^{1-\g}-i_4^{1-\g}
+\dots+i_{K-1}^{1-\g}-i_K^{1-\g}.
$$ 
By the calculation in the Appendix, the RHS of~\eqref{upper.est.moment} is
$$
I+K!\times \frac{N^{K(1+\g)/2}}{c^m (1-\g^2)^m\, m!}.
$$
By the same token,
\begin{align*}
\E\left[Y_{\ms}+...+Y_{N}\right]^K\ge
&I+K!
\sum_{1\le i_1<i_2<\dots < i_K\le N}
\exp\left(dU_{i_{1},...,i_{k}}\right)
=I+ \frac{K!\,N^{K(1+\g)/2}}{d^m (1-\g^2)^m\, m!}
\end{align*}
The reason the remaining terms, collected in $I$, are of lower order is the  following.
Apart form the  already estimated term, in the expansion for 
$\E (Y_{\ms}+\dots+Y_N)^K$ for $r=1,2,\dots,K-1$ we also have to sum up the terms of the type
$$
\E (  Y_{i_1}^{p_1} Y_{i_2}^{p_2}\dots  Y_{i_r}^{p_r}),
\text{ where } \ms\le i_1<\dots<i_r\le N,\  \text{all }p_j\ge 1,\ \text{and }p_1+p_2+\dots +p_r=K.
$$
Since $Y_i=\pm 1$, and thus $Y_i^p=1$ if~$p$ is even and $Y_i^p=Y_i$ if $p$ is odd, it suffices to estimate only the sums
$$
{\mathcal R}(r;\ell_1,\dots,\ell_r;N;K;\g):=\sum\E (  Y_{i_1} Y_{i_2}\dots  Y_{i_r}),
$$
where the summation is taken over all sets $(i_1,\dots,i_r)$ such that $i_{k+1}\ge i_k+\ell_k$, $1\le\ell_k\le K$, for all~$k$, $i_1\ge 1$ and $i_r\le N$. However, since $r\le K-1$, each of the sums ${\mathcal R}(r;\ell_1,\dots,\ell_r;N;K;\g)$ is at most of order $N^{r(1+\g)/2}\le N^{(K-1)(1+\g)/2}$ precisely by the same arguments which ere used to estimate the sum in~\eqref{upper.est.moment}. The number of those sums can be large, as it is the number of integer partitions of~$K$, but it depends only on~$K$ and does not increase with~$N$.

Consequently,  for $m\ge 1$ we have
\begin{align*}
&{\sf (I)}\le\liminf_{l\to\infty}\E\left[\frac{Y_{\ms}+...+Y_{N_{l}}}{N_l^{\frac{1+\g}{2}}}\right]^K \le \limsup_{l\to\infty}\E\left[\frac{Y_{\ms}+...+Y_{N_{l}}}{N_l^{\frac{1+\g}{2}}}\right]^K\le {\sf (II)},
\end{align*}
where
$$
{\sf (II)}:=
\frac{(2m)!}{[c (1-\g^2)]^m \, m!}=\frac{(2m)!}{[2a (1-\g)]^m \, m!}
=\frac{(2m)!}{2^m \, m!} \cdot [a (1-\g)]^{-m}
= \frac{(2m-1) !!}{ [a (1-\g)]^{m}},
$$
and by similar computation, 
$$
{\sf (I)}:= \frac{(2m-1) !!}{ [A (1-\g)]^{m}}.$$
The proof is complete.
\end{proof}

\section{When does the Law of Large Numbers hold for general sequences $\{p_n\}$?}
A natural question to ask is when $S_N$ obeys the Strong (Weak) Law of Large Numbers. The following result gives a partial answer.

For a positive even number $K$, introduce the shorthand 
$$
E(N,K):=N^{-K}\sum_{1\le i_1<i_2<\dots < i_K\le N}e_{i_{1},i_2} e_{i_3,i_4} \dots e_{i_{K-1},i_K},
$$
and note that
$$
\Var\left(S_N\right)=N+2N^2E(N,2),
$$
that is,
$$
\Var\left(S_N/N\right)=1/N+2E(N,2),
$$
\begin{thm}\label{thmSLLN}
${}$
\begin{itemize}
\item[(a)] (Strong Law) Assume that at least one of the following two conditions hold.

{\sf(C1)}  
$$
\sum_N  \frac {E(N,2)}{N} <\infty.
$$

{\sf(C2)} For some even number $K$,
\begin{align}\label{eqSLLN}
\sum_N E(N,K)<\infty.
\end{align}
Then SLLN holds, that is $S_N/N\to 0$ a.s.

\item[(b)] (Weak Law) If for all positive even number $K$,
$$
\lim_{N\to\infty}E(N,K)=0,
$$
then $S_N/N\to 0$ in probability, that is, WLLN holds.

\item[(c)] (no LLN) If for all positive even number $K$,
$$
\exists \lim_{N\to\infty}E(N,K)=:\mu_K>0,
$$
and 
\begin{equation}\label{Carleman}
\sum_{K\ \mathrm{even}}\frac{1}{\mu_{K}^{1/K}}=\infty,
\end{equation}
then the Law of Large Numbers breaks down, and in fact, the law of~$S_N/N$  converges  to a law which  has zero odd moments and even moments~$\{\mu_K\}$.
\end{itemize}
\end{thm}

Note that~\eqref{Carleman} is the so called {\it Carleman-condition}, guaranteeing that the $\mu_K$'s correspond to at most one probability law (see Theorem~3.11, Section~2, in~\cite{DUR}.)

\begin{proof}
We will use the facts about the {\it method of moments} for weak convergence discussed in the proof of Theorem~\ref{gamma.thm}, along with the fact that from~\eqref{even.moments} it follows that
$$
\E \left(\frac{S_N}{N}\right)^K=
K!N^{-K}\sum_{1\le i_1<i_2<\dots < i_K\le N}
\E (Y_{i_1}Y_{i_2}\dots Y_{i_K}) + \O(1/N)
=K!E(N,K)+ \O(1/N).
$$
\begin{itemize}
\item[(a)]
We prove for the two assumptions separately.

Under (C1), the statement follows from Theorem~1 in~\cite{Lyons}, as $e_{i,j}=\mathtt{Cov}(Y_i,Y_j)$.

Under (C2), along the lines of Theorem~6.5 in Section~1 of~\cite{DUR}, we note that for $\eps>0$, one has
$$
\P\left(\left|\frac{S_N}N\right|>\eps\right)\le \frac{\E S_N^K}{\eps^K N^K}
$$
by the Markov inequality (recall that $K$ is even). Since, by~\eqref{even.moments}, the expression on the lefthand side of~\eqref{eqSLLN} is the leading order term in~$\E S_N^K$,
by~\eqref{eqSLLN}, we have $\sum_N \P(|S_N/N|>\eps)<\infty$, and thus, by the Borel-Cantelli lemma, $\P(|S_N/N|>\eps\ \text{i.o.})=0$, which implies the statement.

\item[(b)] Since the deterministically zero distribution is uniquely determined by its moments, convergence in law to that distribution follows from the convergence of all moments to zero. Under the second condition in the theorem, all moments of $S_N/N$ converge to zero (the odd moments are zero by symmetry) and thus $S_N/N$ converge to zero in law (and also in probability, since the limit is deterministic). 

\item[(c)] Assume that the  conditions in {\sf (c)} hold.  Since the moments of $S_N/N$ converge (the odd moments are zero by symmetry), the corresponding laws are tight and, by the Carleman condition, all subsequential limits are the same. That is, as $N\to\infty$, {\sf Law}($S_N/N$)  converges to a law with moments given by $\mu_K$, and since  $\mu_K>0$, the limit cannot be deterministically zero.
\end{itemize}
\end{proof}

\begin{corollary}
When $p_n=a/n^{\g}$ with $0<\g<1$ (sub-critical case),  the Strong Law of Large Numbers holds: $S_N/N\to 0$, $\P$-a.s.
\end{corollary}
\begin{proof}
This follows from Theorem~\ref{thmSLLN}(a), using condition (C1), as we have seen that the inner sum is of order~$N^{1+\gamma}$.

Alternatively, it also follows from Theorem~\ref{thmSLLN}(a), using condition (C2), as we have checked that $\E S_N^K \sim N^{K(1+\g)/2}$, and so we can choose
$\nu_N =\frac{\sf const}{N^{K(1-\g)/2}}$ provided $K$ is even and $K(1-\g)>2$.
\end{proof}

\section{Further heuristic arguments and a conjecture}
Consider a sum of $N\ge 1$ variables having the same law with finite variance.
As is well known, the two `extreme cases' for a sum are the independent case, when the variance is linear and one gets the Central Limit Theorem, and the other one is when all the variables are identical and the variance grows like~$N^2$.
By analogy then, after recalling that in our model
$$
\Var\left(S_N\right)=N+2N^2E(N,2),
$$ 
it seems that the first crucial question is whether
\begin{equation}\label{crucial}
E(N,2)=\O(1/N),\quad N\to\infty
\end{equation}
holds.

In case of \eqref{crucial} fails, one should know whether at least
\begin{equation}\label{second.crucial}
E(N,2)=o(1)
\end{equation}
holds.

Indeed, if~\eqref{crucial} is true, then $\Var\left(S_N\right)$ is of order~$N$, and one expects that CLT holds, that is the fluctuations for the proportion of heads around~$1/2$ is of order~$\sqrt{n}$. This happens when $p_n\equiv p\in (0,1)$. Condition~\eqref{second.crucial} should intuitively be the one that guarantees WLLN to hold.

In light of this, we make the following Conjecture.

\begin{conj}\label{3cases} Let $p_n\in [0,1]$ for $n\ge 1$.
\begin{itemize}
\item[(i)] If \eqref{crucial} holds for $\{p_n\}$, then the proportion of heads obeys CLT. (See Ex0 below.)

\item[(ii)] If \eqref{crucial} fails for $\{p_n\}$, but \eqref{second.crucial} holds, there is a non-standard CLT for the proportion, i.e. the fluctuation about $1/2$ is larger than order $\sqrt{n}$. (See Ex1 below.)

\item[(iii)] If even \eqref{second.crucial} fails for $\{p_n\}$, then WLLN is no longer valid for the proportion, that is the proportion is not concentrated about $1/2$ at all. (See Ex2 and Ex3 below.)
\end{itemize}
\end{conj}
Note that our condition (C1) for the {\it Strong} LLN is more stringent than~\eqref{second.crucial}.

\subsection*{Examples supporting Conjecture \ref{3cases}}

In the examples below, the deviations from the Central Limit Theorem are becoming more marked as we go from Ex1 to Ex2 to Ex3.

\begin{enumerate}
\item[{\bf (Ex0})] (Markov chain CLT) Consider the case $p_n=c$ for all $n\ge 1$, where $0<c<1$. If $c=1/2$, we get an i.i.d.\ sequence of heads/tails and CLT applies. Now assume $c\ne 1/2$. Then the outcomes are not independent. Indeed, denoting $\kappa:=1-2c\in (-1,1)$, we have
$$
e_{i,j}=\kappa^{j-i}
$$
and
\begin{align*}
N^2 E(N,2)
=(N-1)\kappa+(N-2)\kappa^2+\dots+\kappa^{N-1}
=\frac{\kappa (N-1)}{1-\kappa}-\frac{\kappa^2 \left(1-\kappa^{N-1}\right)}{(1-\kappa)^2}.
\end{align*}
Therefore the variance is still of order $N$ but the constant has changed. Recall $\mathtt{Cov} (Y_i,Y_j)=e_{i,j}=\kappa^{j-i}$, and, following~\cite{MarkovChainCLT}, define
$$
\sigma^2:=1+2\sum_{i=1}^{\infty} \mathtt{Cov} (Y_i,Y_j)
=1+2\sum_{i=1}^{\infty} \kappa^i=1+\frac{2\kappa}{1-\kappa}=1+\frac{1-2c}{c}=\frac{1-c}{c},
$$
when $Y_0\sim$~{\sf Bernoulli}$(1/2)$.
In this case, since we are dealing with a time homogeneous Markov chain, it is well known (see~\cite{MarkovChainCLT}) that
$$
{\sf Law}(S_N/\sqrt{N})\to {\sf Normal}\left(0,\frac{1-c}{c}\right).
$$
Hence,
$$
{\sf Law}(\sqrt{N}(\overline{X}_N-1/2))\to {\sf Normal}\left(0,\frac{1-c}{4c}\right),
$$
where $\overline{X}_N:=\frac{X_1+...+X_N}{N}.$
Therefore, only when $c=1/2$, will classical CLT hold for~$Y_i$.
It is also clear that the limiting normal variance can be arbitrarily large when~$c$ is sufficiently small and thus turns occur very rarely.
On the other hand, it can be arbitrarily small if~$c$ is sufficiently close to~$1$ and thus turns occur very frequently.

\item[{\bf (Ex1)}] (CLT breaks down) Consider the case $p_n:=a/n^{\g}$  with $0<\g<1$.
Then
$$
e_{i,j} \cong \exp\left\{-c[j^{1-\g}-i^{1-\g}]\right\}.
$$
Consequently 
$$
N^2 E(N,2)=\sum_{i=1}^{N-1}\sum_{j=i+1}^N e^{-c[j^{1-\g}-i^{1-\g}]}
\cong
\sum_{i=1}^N  \frac{i^\g}{c(1-\g)}=\frac{N^{\g+1}}{c(1-\g)^2},
$$
that is $\mathtt{Var}(S_N)$ is of order $N^{\g+1}$, and the power is strictly between $1$ and $2$.
Hence~\eqref{crucial} is false. The closer~$\g$ to~$1$, the more the situation differs from CLT. However,~\eqref{second.crucial} is true. Thus, the Law of Large Numbers is still in force, so the proportion is still around~$1/2$, but the fluctuations are non-classical (larger than in CLT).

\item[{\bf (Ex2})] (LLN breaks down) Consider the case when $p_n=1/n$. Then
$$
e_{i,j}=\left(1-\frac 2j\right)\left(1-\frac 2{j-1}\right)
\dots \left(1-\frac 2{i+1}\right)=\frac{(i-1)i}{(j-1)j}.
$$
Consequently,
$$
N^2 E(N,2)=\binom{N+1}{2}
$$
is of order $N^2$, that is,~\eqref{crucial} and even \eqref{second.crucial} are false, causing The Law of Large Numbers to break down, and the proportion is no longer around~$1/2$. This means that the correlation is as strong as in the  case of identical variables, and the fluctuations are now of the same order~$N$ as the size of the sum.

Similar is the situation when $p_k=\frac{a}{k}$ with~$a>0$. Instead of being around the~$\delta_{1/2}$ distribution, now one obtains all the ${\sf Beta}(a,a)$ distributions.

\item[{\bf(Ex3)}] (Extreme limit) Consider the case when $\sum_n p_n<\infty$. Then
$$
\liminf_{N\to\infty}E(N,2)>0
$$
must hold (hence~\eqref{crucial} and even~\eqref{second.crucial} are false), because
by a well known theorem (as~$p_k>0$), the infinite product
$\Pi:=\Pi_{k\ge 1} (1-2p_k)$ in this case exists and positive, and so $e_{j,i}\ge \Pi$ implies
$$
N^2 E(N,2) \ge\sum_{1\le i<j\le N}\Pi\ge \binom{N}{2}\Pi.
$$
Then, indeed, as we know, the limit is `extreme': 
${\sf Beta}(0,0)=\frac{1}{2}(\delta_0+\delta_1)$, which is
as far away from~$\delta_{1/2}$  as possible!
\end{enumerate}

\section*{Appendix}
In this appendix we will estimate the quantity
\begin{align*}
Q(\ms,N)&:=\sum_{\ms\le i_1<i_2<\dots < i_K\le N}
\exp\left(c\left[i_1^{1-\g}-i_2^{1-\g}+i_3^{1-\g}-i_4^{1-\g}
+\dots+i_{K-1}^{1-\g}-i_K^{1-\g}\right]\right),
\end{align*} 
for large $N$ and fixed $\ms,\g,c$, with $K=2m,\ m\ge1$, needed for equation~\eqref{upper.est.moment}. The result will immediately follow from the following statement, by plugging $l=m$.
\begin{lemma}
For $l=0,1,\dots,m$ we have
\begin{align}\label{eq:QQ}
Q(\ms,N)=\sum_{\ms+2l\le i_{2l+1}<i_{2l}<\dots <  i_{2m-1}<  i_{2m}\le N}
\exp\left\{c\left[i_{2l+1}^{1-\g}-i_{2l+2}^{1-\g}
+\dots+i_{2m-1}^{1-\g}-i_{2m}^{1-\g}\right]\right\}
\times Z_{l},
\end{align} 
where
\begin{align*}
Z_{l}:=\rho_l\cdot \frac{i_{2l+1}^{(1+\g)l}}{l!c^l(1- \g^2)^l}
+o\left(N^{(1+\g)l}\right),
\end{align*} 
(with the convention $i_{2m+1}\equiv N$) and $\rho_l\to 1$ as $N\to\infty$.
\end{lemma}
\begin{proof}
For $l=0$, this is trivially true. Now assume that we have established~\eqref{eq:QQ} for $l\ge 0$. Then
\begin{align*}
Q(\ms,N)&=\sum_{\ms+2l+2\le i_{2l+3}<i_{2l+4}<\dots <  i_{2m-1}<  i_{2m}\le N}
\exp\left\{c\left[i_{2l+3}^{1-\g}-i_{2l+4}^{1-\g}
+\dots+i_{2m-1}^{1-\g}-i_{2m}^{1-\g}\right]\right\}
\\
&\times
\sum_{\ms+2l\le i_{2l+1}<i_{2l+2}<i_{2l+3}}
\exp\left\{c[i_{2l+1}^{1-\g}-i_{2l+2}^{1-\g}]\right\}
\left[\frac{\rho_l\, i_{2l+1}^{(1+\g)l}}{l!c^l(1- \g^2)^l}
+o\left(N^{l(1+\g)}\right)\right]
\end{align*} 
We now need to estimate the sum in the second line.
First note that the sum 
$$
\sum_{\ms\le i_{2l+1}<i_{2l+2}<N}
\exp\left\{c[i_{2l+1}^{1-\g}-i_{2l+2}^{1-\g}]\right\}
$$
(note that each expression is between $0$ and $1$) can be very well approximated by the corresponding integral, since, whenever $y\le x$, $|\tilde x-x|\le 1$ and $|\tilde y-y|\le 1$, the ratio
$$
\frac{e^{c[\tilde{y}^{1-\g}-\tilde{x}^{1-\g}]}}
{e^{c[y^{1-\g}-x^{1-\g}]}}
$$
is bounded above by $e^{c_1[y^{-\g}+x^{-\g}]}$ where $c_1>0$ is some  constant. Hence, outside of the area where~$x$ and~$y$ are both smaller than $\sqrt{N}$, this constant is very close to~$1$, while the double sum over this area is at most~$N$. Therefore, as $N\to\infty$,
\begin{align*}
&\sum_{\ms\le i_{2l+1}<i_{2l+2}<N}
\exp\left\{c[i_{2l+1}^{1-\g}-i_{2l+2}^{1-\g}]\right\}
=(1+o(1))\, \int_{\ms}^{N} \int_{y}^{N} e^{cy^{1-\g}-cx^{1-\g}}\rmd x \rmd y +\O(N).
\end{align*}
To calculate the inner integral, observe
$$
\int e^{-cx^{1-\g}}\cdot\left[1-\frac{\g x^{\g-1}}{c(1-\g)}\right]\rmd x=-\frac{ x^{\g}e^{-cx^{1-\g}}}  {c(1-\g)}+{\rm const },
$$
implying
\begin{align}\label{eqforR}
R \le \int_y^N e^{-cx^{1-\g}}\rmd x \le  \left[1-\frac{\g y^{\g-1}}{c(1-\g)}\right]^{-1} \times R
\text{ where }
R:=\frac{y^\g e^{-cy^{1-\g}}}{c(1-\g)}
 -\frac{N^\g e^{-cN^{1-\g}}}{c(1-\g)}
\end{align}
(note that $y\ge \ms$). On the other hand,
$$
\int_{\ms}^N R e^{c y^{1-\g}}\rmd y
\le \frac{1}{c(1-\g)}
\int_0^N y^{\g}\rmd y=\frac{N^{1+\g}}{c\,(1-\g^2)}.
$$
Consequently, $(*)=\O(N^{1+\g})$ and hence
\begin{align*}
\sum_{\ms+2l\le i_{2l+1}<i_{2l+2}<i_{2l+3}}
\exp\left\{c[i_{2l+1}^{1-\g}-i_{2l+2}^{1-\g}\right\}
\cdot o\left(N^{l(1+\g)}\right)=o\left(N^{(l+1)(1+\g)}\right).
\end{align*} 

The next step is to compute
\begin{align}\label{eqnextstep}
\sum_{\ms+2l\le i_{2l+1}<i_{2l+2}<i_{2l+3}}
 i_{2l+1}^{(1+\g)l} \times \exp\left\{c\left[i_{2l+1}^{1-\g}-i_{2l+2}^{1-\g}\right]\right\}.
\end{align}
Again, we can approximate this sum by the integral
\begin{align*}
&\int_a^b y^q \int_y^b 
  e^{c[y^{1-\g}-x^{1-\g}]}\rmd x \rmd y
=
\int_a^{N^{\g}} \int_y^b \dots
+\int_{N^{\g}}^b \int_y^b 
y^q  e^{c[y^{1-\g}-x^{1-\g}]}\rmd x \rmd y
\\
&=
\int_{N^{\g}}^b y^q   e^{cy^{1-\g}}\left[\int_y^b 
 e^{-cx^{1-\g}}\rmd x\right] \rmd y
+\O(N^{\g q} \cdot N),
\end{align*}
where $a:=\ms+2l$, $b:=i_{2l+3}$, $q:={(1+\g)l}$. We are allowed to do this since, for $|\tilde x-x|\le 1$ and $|\tilde y-x|\le 1$, the ratio
$$
\frac{{\tilde y}^q e^{c\left[
\tilde{y}^{1-\g}-\tilde{x}^{1-\g}\right]}}
{y^q  e^{c\left[y^{1-\g}-x^{1-\g}\right]}}
$$
is bounded above by 
$\left(1+\frac{2l}{y}\right)e^{c_1(x^{-\g}+y^{-\g}}$
where $c_1>0$ is some  constant. Since $y\ge N^{\g}$, this constant is very close to $1$, hence the double sum in~\eqref{eqnextstep} equals
\begin{align}\label{eqdoublint}
(1+o(1))\,
\int_{N^{\g}}^b y^q   e^{cy^{1-\g}}\left[\int_y^b 
 e^{-cx^{1-\g}}\rmd x\right] \rmd y
+o\left(N^{(l+1)(1+\g)}\right).
\end{align}
From~\eqref{eqforR}, since $y\ge N^{\g}$, we get that the inner integral equals $(1+o(1))\, R$. Therefore, the main expression in~\eqref{eqdoublint}, up to a factor $1+o(1)$, equals
\begin{align*}
&\int_{N^{\g}}^b 
\frac{y^{q+\g}}{c(1-\g)}
 -\frac{y^q N^{\g} e^{cy^{1-\g}-cN^{1-\g}}}{c(1-\g)}
  \rmd y
=\frac{b^{q+1+\g}-N^{q+1+\g}}{(q+1+\g)c(1-\g)}
-\int_{N^{\g}}^b 
\frac{y^q N^{\g} e^{cy^{1-\g}-cN^{1-\g}}}{c(1-\g)}   \rmd y
\\
&=:\frac{i_{2l+3}^{(l+1)(1+\g)}}{(l+1)\cdot c(1-\g)^2}
+o\left(N^{(l+1)(1+\g)}\right)+(**).
\end{align*}
Now the only issue which remains is to show that the integral (**) is of smaller order; then the induction step is finished. To this end, fix some $\g<\theta<1$. Then 
\begin{align*}
\int_{N^{\g}}^b y^q N^{\g} e^{cy^{1-\g}-cN^{1-\g}}   \rmd y
&\le
\int_{0}^{N-N^{\theta}} y^q N^{\g} e^{cy^{1-\g}-cN^{1-\g}}   \rmd y
+\int_{N-N^{\theta}}^N y^q N^{\g} e^{cy^{1-\g}-cN^{1-\g}}   \rmd y
\\
&\le
\int_{0}^{N} y^q N^{\g} e^{c[(N-N^{\theta})^{1-\g}-N^{1-\g}]}   \rmd y
+N^{\theta}\times N^{q+\g}
\\
&\le N^{q+\g} \left[N\times e^{-c(1-\g+o(1))N^{\theta-\g}} +N^{\theta}\right]=o\left(N^{q+1+\g}\right)
=o\left(N^{(l+1)(1+\g)}\right),
\end{align*}
as desired.
\end{proof}

\end{document}